\documentclass[12pt,a4paper,oneside]{amsart}
\usepackage[utf8]{inputenc}

\usepackage{amsmath,amscd,hyperref}
\usepackage{amsfonts}
\usepackage{amssymb}
\usepackage{amsthm}
\usepackage{cite}
\usepackage[margin=1in]{geometry}
\usepackage{mathtools}
\usepackage{footnote}

\numberwithin{equation}{section}

\theoremstyle{plain} \newtheorem{thm}{Theorem}[section]
\newtheorem{lem}[thm]{Lemma}

\newtheorem{cor}[thm]{Corollary}

\newtheorem{claim}[thm]{Claim}

\theoremstyle{definition}

\newtheorem{definition}[thm]{Definition}
\newtheorem{set-up}[thm]{Set-Up}

\newcounter{tmp}

\begin{document}

\title{On Irrationality of Hypersurfaces in $\mathbf{P}^{n+1}$}

\author{Ruijie Yang}
\address{Stony Brook University, Stony Brook, NY 11794}
\email{ruijie.yang@stonybrook.edu}

\begin{abstract}
The purpose of this note is to study various measures of irrationality for hypersurfaces in projective spaces which were proposed recently by \cite{BDELU},\cite{Bastianelli}. In particular, we answer the question raised by Bastianelli that if $X\subset \mathbf{P}^{n+1}$ is a very general smooth hypersurface
of dimension $n$ and degree $d\geq 2n+2$, then $\text{stab.irr}(X)=\text{uni.irr}(X)=d-1$. As a corollary, we prove that $\text{irr}(X\times \mathbf{P}^m)=\text{irr}(X)$ for any integer $m\geq 1$.
\end{abstract}

\maketitle

\section{Introduction}

There has been recent interest in studying measures of irrationality for algebraic varieties \cite{BDELU},\cite{Bastianelli}. For example, given an irreducible projective variety $X$ of dimension $n$, the \textit{degree of irrationality} of $X$ is defined as
\[ \text{irr}(X) := \text{min }\{ \delta>0 \ | \ \exists \text{ degree $\delta$ rational covering } X\dashrightarrow \mathbf{P}^n \}. \]
Therefore $\text{irr}(X)=1$ if and only if $X$ is rational. It was established in \cite{BCD},\cite{BDELU} that if $X\subset \mathbf{P}^{n+1}$ is a very general smooth hypersurface of dimension $n$ and degree $d\geq 2n+1$, then $\text{irr}(X)=d-1$.

By analogy with notions of stable rationality and unirationality, Bastianelli \cite{Bastianelli} introduced two birational invariants measuring the failure of a projective variety to be stably rational or unirational:
\[ \text{stab.irr}(X):=\text{min }\{ \text{ irr}(X\times \mathbf{P}^m) \ | \ m \in \mathbf{N}\};   \]
\[ \text{uni.irr}(X):=\text{min }\{ \text{ irr}(T) \ | \ \exists \text{ a rational covering } T \dashrightarrow X \}.\] 
Thus 
\begin{eqnarray*}
\text{stab.irr}(X)=1 &\iff& \text{ X is stably rational}, \\
\text{uni.irr}(X)=1 &\iff& \text{ X is unirational},
\end{eqnarray*}
and in general one has the inequalities 
\[ \mathrm{uni.irr}(X) \leq \mathrm{stab.irr}(X) \leq \mathrm{irr}(X). \]
It was established by Bastianelli in \cite{Bastianelli} that if $X$ is a very general surface of degree $d\geq 5$, then
\[ \mathrm{stab.irr}(X)=\mathrm{uni.irr}(X)=d-1,\]
and Bastianelli also classified the exceptional cases. Here we extend the computation to hypersurfaces of all dimensions.

In fact we will consider more generally correspondences on $\mathbf{P}^n\times X$. We consider the following birational invariant: 
\[ \text{corr}(X):=\text{min }\{ \text{ deg}(\pi_1) \ | \ \exists \text{ a correspondence } \Gamma \subset \mathbf{P}^n \times X \}.\]
where $\pi_1$ is the first projection map from $\Gamma$ to $\mathbf{P}^n$ and $\Gamma$ is any subvariety of $\mathbf{P}^n \times X$ that both dominates $\mathbf{P}^n$ and $X$.
 

Our first results concern $\text{corr}(X)$:
\begingroup
\setcounter{tmp}{\value{thm}}
\setcounter{thm}{0} 
\renewcommand\thethm{\Alph{thm}}
\begin{thm}
Let $X\subset \mathbf{P}^{n+1}$ be a very general smooth hypersurface of degree $d\geq 2n+2$. Then
\[ \mathrm{corr}(X)=d-1.\]
\end{thm}
\endgroup

Lopez and Pirola \cite[Theorem 1.3]{LP} classified correspondences with null trace (see Def \ref{def:nulltrace}) of minimum degree on smooth hypersurfaces in $\mathbf{P}^3$. Our results can be seen as a partial generalization to higher dimensions: if we restrict ourselves to null trace correspondences on $\mathbf{P}^n \times X$, we can compute their minimal degree.

As in \cite{Bastianelli}, we notice that the study of $\text{uni.irr}(X)$ is equivalent to the study of correspondences on $\mathbf{P}^n \times X$. In particular, we will show that $\text{corr}(X)=\text{uni.irr}(X)$ (cf. Lemma \ref{lem:1}). From this we deduce our second result, which answers the question of \cite{Bastianelli}:

\begingroup
\setcounter{tmp}{\value{thm}}
\setcounter{thm}{1} 
\renewcommand\thethm{\Alph{thm}}
\begin{thm}
Let $X\subset \mathbf{P}^{n+1}$ be a very general smooth hypersurface of degree $d\geq 2n+2$. Then
\[ \mathrm{stab.irr}(X)=\mathrm{uni.irr}(X)=d-1.\]
\end{thm}
\endgroup

In particular, we have the following 
\begingroup
\setcounter{tmp}{\value{thm}}
\setcounter{thm}{2} 
\renewcommand\thethm{\Alph{thm}}
\begin{cor}
Let $X\subset \mathbf{P}^{n+1}$ be a very general smooth hypersurface of degree $d\geq 2n+2$. Then
\[ \mathrm{irr}(X\times \mathbf{P}^m)=\mathrm{irr}(X), \]
for any integer $m\geq 1$.
\end{cor}
\endgroup

Totaro \cite{Totaro} showed that a very general hypersurface $X\subset \mathbf{P}^{n+1}$ of degree $d\geq 2\lceil (n+2)/3 \rceil$ is not stably rational. A couple of years later, Schreieder \cite{SS} found a very striking lower bound where $d\geq \log_2n+2$, which means that most of hypersurfaces are not stably rational. Therefore, one has $\text{stab.irr}(X)>1$ if $d\geq \log_2n+2$. It's interesting to ask further what is the stable irrationality of a degree $d$ hypersurface in this range.

On the other hand, Bastianelli, Ciliberto, Flamini and Supino \cite[Section 5.2]{BCFS} conjectured that 
\[ \text{conn.gon}(X) \leq d-\left \lfloor{\frac{\sqrt{8n+9}-1}{2}}\right \rfloor < d-1=\text{uni.irr}(X).\]
This means that even though it's very hard to determine whether rationally connected varieties are unirational (equivalently whether $\text{conn.gon}(X)=1$ implies $\text{uni.irr}(X)=1$), when $d$ is large these two invariants should capture very different phenomena.

For the proof of Theorem A, we first show that if the degree of a correspondence is less or equal than $d-2$, then one can find on $X$ a relatively large subvariety with bounded covering gonality; this is impossible for very general hypersurface. The method is essentially the same as \cite{BDELU} but the difference is that we work directly on the correspondence instead of passing to the Grassmannian.

In $\S 2$ we discuss some properties of correspondences with null trace and $\S 3$ is devoted to the proof of the main theorems.


We are grateful to Daniele Agostini, Radu Laza, Rob Lazarsfeld, David Stapleton and Zhiwei Zheng for helpful discussions. We want to especially thank Francesco Bastianelli for reading a previous version of this note.

\section{Correspondences}
In this section, we sketch some basic properties of correspondences following \cite{BCD}.

Let $X$ and $Y$ be smooth irreducible complex projective varieties of dimension $n$.

\begin{definition}
An correspondence of $Y$-degree $m$ on $Y\times X$ is a reduced pure $n$-dimensional subvariety $\Gamma \subset Y\times X$ such that the projections $\pi_1:\Gamma \to Y$, $\pi_2:\Gamma \to X$ are generically finite dominant morphisms with $\text{deg}(\pi_1)=m$.
\end{definition}

Recall that for any correspondence $\Gamma \subset Y\times X$, one has \textit{Mumford's trace map} (cf.\cite{BCD}):
\[ \text{Tr}_{X/Y}: H^0(X,K_X) \to H^0(Y,K_Y).\]
In brief, $\text{Tr}_{X/Y}(\omega)=\text{Tr}_{\Gamma/Y}(\pi_2^\ast \omega)$, where $\text{Tr}_{\Gamma/Y}$ is the trace map associated to the generically finite morphism $\Gamma \to Y$.

\begin{definition}\label{def:nulltrace}
A correspondence $\Gamma \subset Y\times X$ has \textit{null trace} to $Y$ if the associated trace map is identically zero. 
\end{definition}


Using the Cayley-Bacharach properties, correspondences with null trace on a smooth hypersurface are analyzed by Bastianelli, Cortini and De Poi in \cite[Theorem 2.5]{BCD}. Their result is
\begin{thm}\label{thm:CB}
Let $X \subset \mathbf{P}^{n+1}$ be a smooth hypersurface of degree $d\geq n+3$ and let 
\[ \Gamma \subset Y\times X \]
be a correspondence of $Y$-degree $m$ with null trace to $Y$. Let $y \in Y$ be a point such that  $\text{dim }\pi_1^{-1}(y)=0$ and let $\pi_1^{-1}(y)=\{ (y,x_i) \in \Gamma \ | \ i=1,\ldots,m \}$ where the $x_i$ are distinct points. Then 
\[ m\geq d-n,\]
and if $m\leq 2d-2n-3$, then the $0$-cycle $Z_y=\sum_{i=1}^m x_i$ lies on a line in $\mathbf{P}^{n+1}$. 
\end{thm}

We will work with the following
\begin{set-up}\label{set-up:1}
Denote by $X \subset \mathbf{P}^{n+1}$ a very general smooth hypersurface of degree $d$, and suppose given a correspondence
$ \Gamma \subset \mathbf{P}^n \times X$ of $\mathbf{P}^n$-degree $m$. 
We assume that 
\[ d\geq 2n+2 \text{ and } m\leq d-2.\]

\end{set-up}


\begin{cor}\label{cor:BCD}
Assume that we are in the situation of \ref{set-up:1}. For general $y \in \mathbf{P}^n$, define $Z_y$ as in the previous theorem. Then we have

\begin{enumerate}
\item $m\geq d-n$. \label{item:1.1}
\item $Z_y$ lies on a line $l_y \subset \mathbf{P}^{n+1}$. \label{item:1.2}
\end{enumerate}
\end{cor}

\begin{proof}
Notice that $\Gamma$ has null trace to $\mathbf{P}^n$ because $H^{0}(\mathbf{P}^n,K_{\mathbf{P}^n})=\{0\}$. Moreover the pair $(d,m)$ satisfies the condition $m\leq 2d-2n-3$. Therefore Theorem \ref{thm:CB} applies.
\end{proof}

\section{Proofs}
In this section, we give the proof of main theorems in the introduction. We will establish Theorem A first.

We assume until the end of the proof of Theorem A that we are in the situation of \ref{set-up:1}. Notice that any rational covering $X\dashrightarrow \mathbf{P}^n$ of degree $\delta$ gives rise to a correspondence of $\mathbf{P}^n$-degree $\delta$ on $\mathbf{P}^n\times X$. Hence by \cite[Theorem C]{BDELU} we have 
\[ \text{corr}(X)\leq \text{irr}(X)=d-1.\]
Therefore it suffices to show that $\text{corr}(X)\geq d-1$ and we will argue by contradiction.

Since we are in the situation of \ref{set-up:1}, by Corollary \ref{cor:BCD} one has a classifying map:
\[ \phi: U \to \mathbf{G}=\mathbf{G}(1,n+1).\]
Here $U$ is the Zariski-open subset of $\mathbf{P}^n$ where the fiber $Z_y=\pi_1^{-1}(y)$ consists of $m$ distinct points. Note that $U$ being open in $\mathbf{P}^n$ is a rational variety itself. Another observation is that $\phi$ is a generically finite map onto its image because $\pi_2:\Gamma \to X$ is generically finite.

Now we have the following diagram:
\[
\begin{CD}
 W'   @>\phi'>>  W @>\mu >> \mathbf{P}^{n+1} \\
@V\pi' VV                                    @V\pi VV \\
U               @>\phi>>       \mathbf{G}
\end{CD}
\]
Here $\pi:W \to \mathbf{G}$ is the tautological $\mathbf{P}^1$-bundle on $\mathbf{G}$, $\mu:W \to \mathbf{P}^{n+1}$ is the evaluation map and 
\[ W'=_{\text{def}}\phi^\ast W\]
is the pullback of $W$ via the classifying map $\phi$.

\begin{claim}\label{claim:1}
$W'$ is an irreducible $n+1$-dimensional variety and $\psi=_{\text{def}} \mu \circ \phi'$ is dominant onto $\mathbf{P}^{n+1}$.
\end{claim} 

\begin{proof}
Notice that $\pi':W'\to U$ is a $\mathbf{P}^1$-bundle and $U$ is irreducible, so $W'$ must be irreducible. Since $\text{dim }\psi(W')\leq n+1$, it suffices to show that $\psi$ is dominant. We prove this by contradiction. Suppose $\text{dim }\psi(W')\leq n$. Since $\Gamma \to X$ is dominant and an open subset of $\Gamma$ is contained in $W'$ by Corollary \ref{cor:BCD}, this would imply that $X$ contains $\psi(W')$ as an open subset. 
Therefore $X$ is uniruled, but this is impossible since $\text{deg}(X)$ is greater than $n+1$.
\end{proof}

\begin{proof}[Proof of Theorem A]
Recall that we are in the situation of \ref{set-up:1}, where $\Gamma \subset \mathbf{P}^n\times X$ is a correspondence of $\mathbf{P}^n$-degree $m\leq d-2$ by contradiction. Define $\Gamma'$ to be the restriction of $\Gamma$ to $U\times \mathbf{P}^n$. By Corollary \ref{cor:BCD}, $\Gamma'$ is a divisor in $W'$ of relative degree $m$ over $U$. Let $X'$ be the full pre-image of $X$ in $W'$ so that $X'$ is a divisor in $W'$ of relative degree $d$ over $U$. We can write 
\[ X'=\Gamma' + F, \]
where $F$ is a divisor of relative degree $d-m\geq 2$ over $U$. Now fix any irreducible component $R\subset F$ that dominates $U$ and view $R$ as a reduced irreducible variety of dimension $n$. Thus $R$ sits in a diagram
\begin{eqnarray} \begin{CD}
X @<<< R \\
@VVV   @VVV\\
\mathbf{P}^{n+1} @<<\psi=\mu\circ \phi'< W' @>>> U\\
\end{CD}
\end{eqnarray}
and we have
\begin{eqnarray} 0< e=_{\text{def}}\text{deg}(R\to U) \leq d-m.\end{eqnarray}

Put 
\begin{eqnarray} S=_{\text{def}} \psi(R) \subset X, \end{eqnarray}
and let $s=\text{dim} S$. Suppose first that $s=0$, i.e. $S$ consists of a single point $p\in X$. But this would imply $\text{deg}(\Gamma \to U)=d-1$, which contradicts with our assumption. Therefore we may assume that $1\leq s\leq n-1$.

Note next that $\text{cov.gon}(S)\leq e$. In fact, one can choose a rational subvariety $L\subset U$ of dimension $s$ with the property that an irreducible component $R^\ast \subset R$ of the inverse image of $L$ in $R$ is generically finite over $S$. Since $\text{deg}(R^\ast \to L)\leq e$, and since $L$ is rational, we see that $\text{cov.gon}(R^\ast)\leq e$. Hence \cite[lemma 1.9]{BDELU} applies to show that $\text{cov.gon}(S)\leq e$.

Now denote by $K_{W'/\mathbf{P}^{n+1}}$ the relative canonical bundle of $\psi$, and consider a general fiber $l=l_y$ of $(W'\to U)$. We assert: \footnote{Notice that even though we are working on an open variety, this intersection product still makes sense because we are intersecting a divisor with the fiber of a proper map.} \footnote{Bastianelli pointed out to me that it is possible to avoid this assertion by passing to the Grassmannian and argue as in \cite{BDELU}.}
\begin{eqnarray}\label{eqn:intersection}
l\cdot K_{W'/\mathbf{P}^{n+1}}=n.\end{eqnarray} 
Grant this for now. 
Since $\text{dim }\psi(R)=s$, by \cite[Corollary A.6]{BDELU} we have 
\begin{eqnarray} \text{ord}_R(K_{W'/\mathbf{P}^{n+1}})\geq n-s.\end{eqnarray}
Hence we must have
\[ n=l\cdot K_{W'/\mathbf{P}^{n+1}}\geq \text{ord}_R(K_{W'/\mathbf{P}^{n+1}})\cdot l\cdot R\geq (n-s)\cdot\text{deg}(R\to U)\geq (n-s)e.\]
Now recall that we assume $s\geq 1$. Then it follows from the computations of Ein and Voisin \cite[Proposition 3.8]{BDELU} that 
\[ e\geq \text{con.gon}(S)\geq d-2n+s.\]
One finds that 
\[ d \leq 2n-s+e \leq 2n-s+\frac{n}{n-s}\leq 2n+1.\]
which is impossible since $d\geq 2n+2$.

It remains to prove (\ref{eqn:intersection}). We consider the restriction of the tangent map $T_{W'} \to \psi^\ast T_{\mathbf{P}}$ to $l\cong \mathbf{P}^1$. By the Euler sequence, one has 
\[ \psi^\ast T_{\mathbf{P}}|_l\cong \mathcal{O}_{\mathbf{P}^1}(1)^{\oplus n}\oplus \mathcal{O}_{\mathbf{P}^1}(2). \]
For $T_{W'}|_l$, we have the following exact sequence:
\[ 0 \to T_{W'/U}|_l \to T_{W'}|_l \to \pi'^\ast T_U |_l \to 0. \]
The first term is isomorphic to $T_l\cong \mathcal{O}_{\mathbf{P}^1}(2)$, and the third term is isomorphic to $\mathcal{O}_{\mathbf{P}^1}^{\oplus{n}}$. Notice that this exact sequence of vector bundles splits because
\[ \text{Ext}_{\mathbf{P}^1}(\mathcal{O}_{\mathbf{P}^1}^{\oplus{n}},\mathcal{O}_{\mathbf{P}^1}(2))\cong H^1(\mathbf{P}^1,\mathcal{O}(2))^{\oplus{n}}=\{0\}.\]
Hence we have 
\[ T_{W'}|_l \cong \mathcal{O}_{\mathbf{P}^1}^{\oplus{n}} \oplus \mathcal{O}_{\mathbf{P}^1}(2).\]
Therefore the restriction of the tangent map to $l\cong \mathbf{P}^1$ becomes 
\[ \mathcal{O}_{\mathbf{P}^1}^{\oplus{n}} \oplus \mathcal{O}_{\mathbf{P}^1}(2) \to \mathcal{O}_{\mathbf{P}^1}(1)^{\oplus n}\oplus \mathcal{O}_{\mathbf{P}^1}(2),\]
whose degeneracy locus is thus given by a linear form of degree $n$ on $\mathbf{P}^1$. Since a general fiber doesn't lie in the ramification locus, we must have 
\[ l \cdot K_{W'/\mathbf{P}^{n+1}}=n.\] 
\end{proof}

Now we turn to the proof of Theorem B. We first establish a lemma connecting $\text{corr}(X)$ and $\text{uni.irr}(X)$.

\begin{lem}\label{lem:1}
Let $X$ be an irreducible smooth projective variety, then 
\[ \mathrm{uni.irr}(X) = \mathrm{corr}(X).\]
\end{lem}

\begin{proof}
Let $T$ be a smooth $n$-dimensional variety with two dominant rational maps 
\[ f:T\dashrightarrow \mathbf{P}^n, g:T \dashrightarrow X. \] 
By considering the closure of the graph of $f$ and $g$, we see that $T$ maps onto a correspondence $\Gamma \subset \mathbf{P}^n\times X$ and $\text{deg}(f)$ is a multiple of $\text{deg}(\Gamma \to \mathbf{P}^n)$. Hence $\text{uni.irr}(X)\geq \text{corr}(X)$. The other inequality is obvious.
\end{proof}

\begin{proof}[Proof of Theorem B]

By Lemma \ref{lem:1} and Theorem A, one has
\[ \text{uni.irr}(X)=\text{corr}(X)=d-1.\]
On the other hand, by \cite[Theorem C]{BDELU} and \cite[Lemma 2.2]{Bastianelli}
\[ d-1=\text{irr}(X)\geq \text{stab.irr}(X) \geq \text{uni.irr}(X),\]
we conclude that $\text{stab.irr}(X)= \text{uni.irr}(X)=d-1$.
\end{proof}

\bibliographystyle{abbrv}
\bibliography{irrationality}{}

\end{document}